\documentclass[11pt]{amsart}

\usepackage{amsmath,amsthm,amssymb,amsfonts}
\usepackage{color}
\usepackage{hyperref}

\usepackage{amscd}
\usepackage[latin1]{inputenc}
\DeclareMathAlphabet{\mathpzc}{OT1}{pzc}{m}{it}
\usepackage{bbm}

\numberwithin{equation}{section}

\swapnumbers
\newtheorem{thm}[subsection]{Theorem}

\newtheorem*{cor*}{Corollary}
\newtheorem{lemma}[subsection]{Lemma}

\newtheorem{propos}[subsection]{Proposition}

\newtheorem*{thm*}{Theorem}
\newtheorem*{thma*}{Theorem A}
\newtheorem*{thmb*}{Theorem B}
\newtheorem*{thmc*}{Theorem C}

\newtheorem{conj}{Conjecture}

\newcounter{consta}

\newcounter{constb}

\newcounter{constc}

\newcounter{constD}

\def\bbr{\mathbb{R}}

\def\bbh{\mathbb{H}}
\def\bbn{\mathbb{N}}

\def\R{\bbr}

\def\Rfrak{\mathfrak{R}}
\def\Ifrak{\mathfrak I}

\DeclareMathOperator\Mod{Mod}
\DeclareMathOperator\area{area}

\DeclareMathOperator\MC{Mod}

\def\SL{{\rm{SL}}}

\def\GL{{\rm{GL}}}

\def\Sp{{\rm{Sp}}}

\def\stab{{\rm stab}}

\def\vare{\varepsilon}

\def\zg0{Z_{G_\omega}(s)}

\def\zg{Z_G(s)}

\def\be{\begin{equation}}
\def\ee{\end{equation}}

\def\dist{{\rm dist}}

\def\Re{\Rfrak}
\def\Im{\Ifrak}

\def\cM{\mathcal M}
\def\cH{\mathcal H}
\def\cT{\mathcal T}
\def\cN{\mathcal N}
\def\tcM{\widehat{\cM}}

\def\Tc{V}

\def\tTc{\widehat{\Tc}}
\def\stab{{\rm Stab}}
\def\zed{\mathbb Z}
\def\cx{\mathbb C}
\def\reals{\mathbb R}
\def\cross{\times}
\def\noz{m}

\def\Teich{Teichm\"uller }
\def\I{\mbox{i}}

\begin{document}
\title{Benjamini-Schramm convergence of periodic orbits}

\author{Amir Mohammadi}
\address{A.M.\ Mathematics Department, UC San Diego}
\email{ammohammadi@ucsd.edu}
\thanks{A.M.\ acknowledges support by the NSF}

\author{Kasra Rafi}
\thanks{K.R.\ acknowledges support by NSERC Discovery grant, RGPIN 06486}
\address{K.R.\ Department of Mathematics, University of Toronto}
\email{kasra.rafi@math.toronto.edu}

\maketitle

\begin{abstract}
We prove a criterion for Benjamini-Schramm
convergence of periodic orbits of Lie groups. This general observation 
is then applied to homogeneous spaces and the space of translation surfaces.  
\end{abstract}

\section{Benjamini-Schramm convergence}\label{sec: BS conv general}
Let $H\subset\SL_N(\mathbb R)$ be a non-compact semisimple group. 
Even though $H\subset \SL_N(\mathbb R)$, we will write $e$ 
for the identity element in $H$. The notation $I$ (for the identity matrix) 
will only be used when the vector space structure of the space of matrices is relevant.  

Let $\|\;\|$ denote the maximum norm on ${\rm Mat}_N(\bbr)$ with respect to the standard basis, and put
\[
B^H(e,R)=\{h\in H: \|h-I\|< R\text{ and }\|h^{-1}-I\|< R\}.
\]

We also equip $H$ with the right invariant Riemannian metric induced by the killing form, 
and let $B^H_{\rm Rie}(e,r)$ denote the ball of radius 
$r$ centered at the identity with respect to this metric. Then for every $R$, there exists $r>0$ so that 
\[
B^H_{\rm Rie}(e,r)\subset B^H(e,R).
\] 
Let $r(R)$ denote $1/2$ the supremum of all such $r$, then $B^H_{\rm Rie}(e,r(R))\subset B^H(e,R)$ and $r(R)\to \infty$ as $R\to \infty$;  indeed it is not difficult to see that $r(R)\geq C \log R$ where $C>0$ depends on the embedding $H\subset \SL_N(\bbr)$. 

Let $\Delta\subset H$ be a discrete subgroup. 
The injectivity radius of $y\in H/\Delta$ is define as the supremum over all $r>0$ so that 
the map $h\mapsto hy$ is injective on $B^H_{\rm Rie}(e,r)$.
  
Let $\Delta_n\subset H$ be a sequence of lattices in $H$. 
The sequence $\{H/\Delta_n: n\in\bbn\}$ {\em Benjamini-Schramm} converges to $H$ if 
for every $r>0$ we have 
\[
\mu_n\Bigl(\{y\in H/\Delta_n: \text{injectivity radius of $y$}<r\}\Bigr)\to 0\quad\text{as $n\to\infty$}
\]  
where $\mu_n$ denote the $H$-invariant probability measure on $H/\Delta_n$ for every $n$.

Throughout, we assume that $H$ acts continuously on $X$ preserving the measure $\mu$; 
also assume that $\stab_H(x)$ is discrete for every $x\in X$.

An orbit $Hx\subset X$ is called {\em periodic} if $Hx\subset X$ is a closed subset 
and $\stab_H(x)$ is a lattice in $H$.

For a periodic orbit $Hx$, let $\mu_{Hx}$ denote the pushforward 
of the $H$-invariant probability measure of $H/\stab_H(x)$ to $Hx$. 

\begin{propos}\label{prop:axiomatic-BS}
Let $\{Hx_n: n\in\bbn\}$ be a sequence of periodic orbits in $X$ satisfying that
\be\label{eq:weak-conv}
\mu_{Hx_n}\to \mu\quad\text{ as $n\to\infty$.}
\ee
Assume further that for every $R>0$ there exists a continuous function 
$f_R:X\to[0,\infty)$ satisfying the following two properties: 
\begin{enumerate}
\item $f_R(x)>0$ for $\mu$-a.e.\ $x\in X$,
\item if $f_R(x)>0$ for some $x\in X$, then $\stab_H(x)\cap B^H(e,R)=\{e\}$.
\end{enumerate}
Then $H/\stab_H(x_n)$ Benjamini-Schramm converges to $H$.   
\end{propos}

\begin{proof}
Let $R>0$. 
Let $Y=Hx\subset X$ be a periodic orbit, and put $\Delta=\stab(x)$.
The map $h\Delta\mapsto hx$ is a homeomorphism from $H/\Delta$ onto $Y$. 
Let $h\Delta\in H/\Delta$, and write $y=hx\in Y$. Suppose now that $h_1h\Delta=h_2h\Delta$ 
for some $h_1,h_2\in B^H(e,R)$.  
Then $\|h_2^{-1}h_1-I\|< NR^2$ and 
\[
h_2^{-1}h_1\in h\Delta h^{-1}=\stab_H(y).
\]

This and the assumption~(2) in the proposition imply that
\be\label{eq:inj-fR}
\text{If $f_{NR^2}(y)>0$, then the injectivity radius of $h\Delta$ is at least $r(R)$;}
\ee
recall that $B^H_{\rm Rie}(e,r(R))\subset B^H(e,R)$. 

Let now $\vare>0$. In view of our assumption~(1) in the proposition, 
there exists a compact subset $K_\vare\subset X$ so that 
\[
\text{$\mu(K_\vare)>1-\vare$ and $f_{NR^2}(x)>0$ for all $x\in K_\vare$.}
\] 

Since $f$ is continuous and $K_\vare$ is compact, there exists some $\delta>0$ so that 
$f_{2R}(x)>0$ for all $x\in \mathcal N_\delta(K_\vare)$, where $\mathcal N_\delta(K_\vare)$ 
denotes a finite open covering of the set $K_\vare$ with balls of radius $\delta$ centered at points in $K_\vare$.

Since $\mathcal N_\delta(K_\vare)$ is an open set and $\mu_{Hx_n}\to \mu$, we conclude that  
\[
\liminf_n\mu_{Hx_n}\bigl(\mathcal N_\delta(K_\vare)\bigr)\geq \mu\bigl(\mathcal N_\delta(K_\vare)\bigr)\geq 1-\vare.
\]
This and the fact that $\mathcal N_\delta(K_\vare)\subset\{y\in Hx_n: f_{NR^2}(y)>0\}$ imply: 
there exists some $n_0$ so that 
\[
\mu_{Hx_n}\Bigl(\{y\in Hx_n: f_{NR^2}(y)>0\}\Bigr)>1-2\vare\quad\text{for all $n>n_0$ }.
\] 

In consequence, using~\eqref{eq:inj-fR} we deduce that
\[
\mu_{H/\stab(x_{n})}\Bigl(\{y\in H/\stab(x_{n}): \text{injectivity radius of $y$ is $< r(R)$}\}\Bigr)<2\vare
\]
for all $n>n_0$. Since $r(R)\to\infty$ as $R\to \infty$, the claim follows.
\end{proof}

In subsequent sections, we discuss two settings where Proposition~\ref{prop:axiomatic-BS}
is applicable: the homogeneous setting is discussed in \S\ref{sec:homog-case} and 
the space of Abelian differentials in \S\ref{sec: Abelian notation}; see in particular Theorems~\ref{thm:geod-planes} and~\ref{thm:main}.

\subsection*{Acknowledgement} 
We would like to thank A.~Eskin, T.~Gelander, C.~Leininger, G.~Margulis, H.~Oh, and Alex Wright for helpful conversations.

\section{Homogeneous spaces}\label{sec:homog-case}
Let ${\bf G}$ be a connencted algebraic group defined over $\bbr$, and let $G={\bf G}(\bbr)^\circ$ 
be the connected component of the identity in the Lie group ${\bf G}(\mathbb R)$. 

Let $\Gamma\subset G$ be a lattice.
Throughout this section, we assume that {\em $\Gamma$ is torsion free}.  
Let $X=G/\Gamma$, and let $\mu_X$ denote the $G$-invariant probability measure on $X$. 

\begin{thm}\label{thm:BS-homog}
Let the notation be as above. 
Let $H\subset G$ be a connected semisimple Lie group. Assume that
\be\label{eq:H-not-normal}
\text{$\bigcap_{g\in G}gHg^{-1}$ is a finite group.}
\ee
Let $\{Hx_n: n\in\bbn\}$ be a sequence of periodic $H$-orbits in $X$ so that
\begin{enumerate}
\item There exists a compact subset $K\subset X$ with $Hx_n\cap K\neq\emptyset$ for all $n$. 
\item For every $H\subset L\subset G$ and any closed orbit $Lx$, at most finitely many of 
the orbits $Hx_n$ are contained in $Lx$.
\end{enumerate}
Then $H/\stab_H(x_n)$ Benjamini-Schramm converges to $H$.
\end{thm}

Note that the condition {\em $\cap_{g\in G}gHg^{-1}$ is a finite group}
in the theorem is satisfied for instance if $G$ semisimple and $H$ does not contain 
any of the simple factors of $G$. 

\begin{thm}\label{thm:geod-planes}
Let $M$ be a real or complex hyperbolic $d$-manifold with $d\geq 3$. 
Assume that $M$ contains infinitely many properly immersed totally geodesic 
hypersurfaces $\{\Tc_n:n\in\bbn\}$. Then $\{\Tc_n\}$ Benjamini-Schramm converges to 
$\mathbb H^{d-1}$ in the real hyperbolic case and to $\mathbb {CH}^{d-1}$ in the complex case.
\end{thm}

\begin{proof}
We prove the result for the case real hyperbolic manifold, the complex case is similar.
 
Let $G={\rm SO}(d,1)^\circ$, $\Gamma=\pi_1(M)$, and $H={\rm SO}(d-1,1)^\circ$. 
Then $\Tc_n$ lifts to a closed orbit $Hx_n$ in $X=G/\Gamma$ for every $n$.

Note that $H\subset G$ is a maximal connected subgroup which is not a parabolic subgroup of $G$. 
Therefore, the assumptions in Theorem~\ref{thm:BS-homog} are satisfied for $G$, $H$, and the orbits $\{Hx_n:n\in\bbn\}$.
The claim thus follows from Theorem~\ref{thm:BS-homog}.
\end{proof}

We note that when $\Gamma$ is arithmetic Theorem~\ref{thm:BS-homog} follows from the work~\cite{7-Samurai}. 
This condition holds if $\Gamma$ is an irreducible lattice and the real rank of $G$ is at least two by Margulis' arithmeticity theorem~\cite{Margulis-Arith-Inv}. Moreover, it was proved by Corlette and Gromov-Shoen~\cite{Corl, Grom-Sch} that lattices in 
${\rm SP}(n,1)$ and $F_4^{-20}$ are arithmetic.
While non-arithmetic lattices in ${\rm SO}(n,1)$, for all $n$, and ${\rm SU}(n,1)$, for $n=2,3$, exist, 
recent developments,~\cite{MM, BFMS, BFMS-Complex}, show that the presence of infinitely 
many totally geodesic hyperplanes\footnote{The works~\cite{BFMS, BFMS-Complex} are indeed more general and allow for properly immersed maximal totally geodesic submanifolds of dimension at least $2$.} in real and complex hyperbolic manifolds of finite volume imply arithmeticity of their fundamental group. Therefore, in all interesting cases, 
the assertion of Theorem~\ref{thm:BS-homog} can be obtained by combining existing rather deep results in the literature. 
However, the proof we provide is different and is arguably simpler. In particular, our proof does not rely on arithmeticity of $\Gamma$,
and relies only on a special case of a equidistribution theorem of Mozes and Shah~\cite{Mozes-Shah}.

\begin{lemma}\label{lem:trivial-stab}
Let the notation and the assumptions be as in Theorem~\ref{thm:BS-homog}. 
Then for $\mu_X$-a.e.\ $x\in X$ we have 
\[
\stab(x)\cap H=\{e\}
\]
\end{lemma}

\begin{proof}
Let $\bf H$ denote the Zariski closure of $H$ in $\bf G$.
Since $H$ is a connected semisimple Lie group, it has finite index in the group $H':={\bf H}(\bbr)\cap G$.
 
By Chevalley's theorem, there exists a finite dimensional (real) representation $(\rho, W)$ 
of $\bf G$ and a vector $w\in W$ so that ${\bf H}=\{g\in {\bf G}: gw=w\}$. In particular, we conclude that 
\be\label{eq:Chevalley}
H'=G\cap {\bf H}=\{g\in G: gw=w\}.
\ee

Let now $x=g_0\Gamma$. Then $\stab(x)=g_0\Gamma g_0^{-1}$, and 
$H\cap g_0\Gamma g_0^{-1}$ is nontrivial if and only if there exists some $e\neq\gamma\in\Gamma$ 
so that $\gamma\in g_0^{-1}Hg_0$. Since $H\subset H'$, we conclude that $\gamma g_0^{-1}w=g_0^{-1}w$. Hence,  
\[
g_0^{-1}\in {\bf F}_\gamma=\{g\in {\bf G}: \gamma gw=gw\}.
\] 

For every $\gamma\in\Gamma$, the set ${\bf F}_\gamma$ is an algebraic variety defined over $\bbr$. 
Moreover, $G={\bf G}(\bbr)^\circ$ is Zariski dense in ${\bf G}$. These and the fact that $\Gamma$ is countable imply that 
unless there exists some $e\neq\delta\in \Gamma$ so that 
\[
\text{$\delta gw=gw\quad$ for all $g\in G$,} 
\]
the lemma holds --- indeed in that case $G\setminus \Bigl(\cup_{\gamma\in\Gamma} {\bf F}_\gamma\Bigr)$ is a conull subset of $G$, 
and for every $g$ in this set we have $H\cap \stab(g\Gamma)=\{e\}$.

Assume now to the contrary that $G=\{g\in G: \delta gw=gw\}$ for some nontrivial 
$\delta\in\Gamma$. Then by~\eqref{eq:Chevalley} we have $\delta\in gH'g^{-1}$ for all $g\in G$, hence,
\[
\delta\in\bigcap_{g\in G} gH'g^{-1}.
\]

Since $[H':H]<\infty$, there exists some $n$ so that $\delta^n\in gHg^{-1}$ for all $g\in G$. That is,
$\delta^n\in\cap_{g\in G}\, gHg^{-1}$. However, $\Gamma$ is torsion free and $\cap_{g\in G}\, gHg^{-1}$ is a finite group.
This contradiction completes the proof.
\end{proof}

\begin{proof}[Proof of Theorem~\ref{thm:BS-homog}]
We may and will assume that $G\subset \SL_N(\bbr)$ for some $N$. 
As before, for all subgroups $L\subset G$ and all $R>0$, let 
\[
B^L(e, R)=\{g\in L: \|g-I\|< R \text{ and }\|g^{-1}-I\|< R\}
\]
where $\|\;\|$ denotes the maximum norm on $\SL_N(\bbr)$ with respect to the standard basis.  

Recall that $\mu_X$ denotes the $G$-invariant probability measure on $X$.
First note that by a theorem of Mozes and Shah~\cite{Mozes-Shah} and our 
assumptions~(1) and~(2) in the theorem, we have
\be\label{eq:muHxn-to-muX}
\mu_{Hx_n}\to \mu_X\quad\text{as $n\to\infty$}.
\ee

Let $\dist$ denote the right invariant Riemannian metric on $G$ induced using the killing form. 
Let $R>1$, and put ${\rm Stab}(x)_R={\rm Stab}(x)\cap B^G(e,R)$; this is a finite set. 
Define $f_R:X\to[0,\infty)$ by 
\[
f_R(x)=\dist_{\rm H}\Big(\bar{B}^H(e,R),\bigl({\rm Stab}(x)_R\setminus \{e\}\bigr)\Bigr)
\] 
where $\dist_{\rm H}$ is the Hausdorff distance and $\bar{B}^H(e,R)$ is the closure of $B^H(e,R)$. 

Since $\stab(g\Gamma)=g\Gamma g^{-1}$ and $R$ is fixed, $f_R$ is continuous. 
Furthermore, $f_R(x)>0$ for some $x\in X$ if and only if ${B}^H(e,R)\cap {\rm Stab}(x)=\{e\}$.
In particular, by Lemma~\ref{lem:trivial-stab} we have 
\[
\text{$f_R(x)>0\quad$ for $\mu_X$-a.e.\ $x\in X$.}
\]
Altogether, we deduce that $f_R$ satisfies the conditions in Proposition~\ref{prop:axiomatic-BS}.

The theorem thus follows from Proposition~\ref{prop:axiomatic-BS} in view of~\eqref{eq:muHxn-to-muX}. 
\end{proof}

\section{The space of Abelian differentials}\label{sec: Abelian notation}
Let $g\geq 2$, and let $\mathcal T_g$ denote the Teichm\"{u}ller space of complex structure on a compact Riemann 
surface of genus $g$. We denote by $\cM_g$ the corresponding moduli space, i.e., the quotient of $\mathcal T_g$ by 
the mapping class group, $\MC_g$. 

As it is well-known, $\MC_g$ is not torsion free, however, it has subgroups of finite index which are torsion free --- indeed
the kernel of the natural map from $\MC_g$ to $\Sp_{2g}(\zed/3\zed)$ is torsion free.

We fix, once and for all, a covering map 
\[
\tcM_g\to \cM_g
\] 
which corresponds to a torsion free finite index subgroup of $\MC_g$.

Let $f:\bbh^2\to \cM_g$ be an isometric immersion for the Teichm\"{u}ller metric.  
Typically, $f(\bbh^2)$ is dense in $\cM_g$, however, there are situations where $f(\bbh^2)$ is an algebraic curve
in $\cM_g$. In the latter case, the stabilizer $\Delta$ of $f$ is a lattice in ${\rm Isom}(\bbh^2)$, and we obtain a {\em Teichm\"uller curve}   
\[
f:\Tc=\bbh^2/\Delta\to\cM_g.
\]

For every $g\geq 2$, the moduli space $\cM_g$ contains a dense family of Teichm\"{u}ller curves 
which arise as branched cover of flat tori. 
There are also examples of infinite families of {\em primitive} Teichm\"{u}ller curves, i.e., Teichm\"{u}ller curves which do not arise as a branched cover of flat tori, in $\cM_g$ when $g=2,3,4$,~\cite{McMullen-SL2, MMW-TeichCurves}.


\begin{thm}\label{thm:main}
Let $\{\Tc_n:n\in\bbn\}$ be an infinite family of Techim\"uller curves in $\cM_g$.
For every $n$, let $\tTc_n\to\Tc_n$ be a lift of $\Tc_n$ to $\tcM_g$.
Then $\{\tTc_n:n\in\bbn\}$ Benjamini-Schramm converges to $\mathbb H^2$.
\end{thm}

C.~Leininger and A.~Wright (independently) have supplied an alternative 
(and arguably softer) proof of Theorem~\ref{thm:main}.
This argument relies on the fact that the length of shortest geodesic on Teichm\"uller curves tends to infinity, see Proposition~\ref{prop:stabilizer-H-alpha}, and is independent of measure classification theorems. 
We also thank T.~Gelander for helpful communications regarding IRSs.

Here, we present a proof based on Proposition~\ref{prop:axiomatic-BS} and~\cite{EMM-Orbit} 
to highlight a unifying theme between the homogeneous setting and the setting at hand.



\medskip

For every $M\in\cT_g$, let $\Omega(M)$ be the $g$-dimensional space of 
holomorphic 1-forms on $M$. By integrating a non-zero form $\omega\in\Omega(M)$ we obtain, away from the zeros of $\omega$, a flat metric $|\omega|$ on $M$ and local charts whose transition functions are translations.   

Form a vector bundle over the Teichm\"{u}ller space $\cT_g$ where the fiber over each point is $\Omega(M)$. 
Let $\Omega\mathcal T_g\to\mathcal T_g$ be the complement of the zero section of this  
vector bundle.

There is a natural action of ${\rm GL}^+_2(\reals)$ (and hence of $\SL_2(\reals)$) on $\Omega\mathcal T_g$:
given a holomorphic $1$-form $\omega=\Re(\omega)+i\Im(\omega)$ and 
$h =\begin{pmatrix} a & b \\ c & d \end{pmatrix} \in{\rm GL}^+_2(\reals)$,
\be\label{eq:SL2R-OmegaT}
h\cdot\omega=\begin{pmatrix} i \\ i \end{pmatrix}\begin{pmatrix} a & b \\ c & d \end{pmatrix}\begin{pmatrix} \Re(\omega) \\ \Im(\omega) \end{pmatrix}.
\ee

We let $\Omega\cM_g\to\cM_g$ denote the quotient of 
$\Omega\cT_g$ by action of the mapping class group of $S_g$.

For every $\alpha=(\alpha_1,\ldots,\alpha_{\noz})$ with $\sum\alpha_i=2g-2$, 
let $\cH(\alpha)$ denote the set of $(M,\omega)\in\Omega\cM_g$ where $\omega$ has zeros of type $\alpha$.
Then $\Omega\cM_g=\bigsqcup \cH(\alpha)$.

Let $(M,\omega)\in\cH(\alpha)$ and let $\Sigma \subset M$ denote the set of zeroes of $\omega$. Let
$\{ \gamma_1, \dots, \gamma_k\}$ denote a $\zed$-basis for
the relative 
homology group $H_1(M,\Sigma, \zed)$. {}{(It is
  convenient to assume that the basis is obtained by extending a
  symplectic basis for the absolute homology group $H_1(M,\zed)$.)}
We can define a map $\Phi:
\cH(\alpha) \to \cx^k$ by 
\begin{displaymath}
\Phi(M,\omega) = \left( \int_{\gamma_1} \omega, \dots, \int_{\gamma_k}
  w \right)
\end{displaymath}
The map $\Phi$ (which depends on a choice of the basis $\{ \gamma_1,
\dots, \gamma_k\}$) is a local coordinate system on $(M,\omega)$.  
Alternatively,
we may think of the cohomology class $[\omega] \in H^1(M,\Sigma, \cx)$ as a
local coordinate on the stratum $\cH(\alpha)$. We
will call these coordinates {\em period coordinates}. 

The area of a translation surface is given by 
\begin{displaymath}
a(M,\omega) = \frac{i}{2} \int_M \omega \wedge \bar{\omega}.
\end{displaymath}
We let $\Omega_1\mathcal M_g$  and $\cH_1(\alpha)$ denote the locus of unit area $1$-forms in $\Omega\mathcal M_g$ and $\cH(\alpha)$, respecitively.

\subsection*{The $\SL_2(\reals)$-action and the Kontsevich-Zorich cocycle}
The action in \eqref{eq:SL2R-OmegaT} descends to an action of $\SL_2(\reals)$ on $\cH_1(\alpha)$.
Indeed, write $\Phi(M,\omega)$ as a $2 \cross d$ matrix $x$. The action of 
$\SL_2(\reals)$ in these coordinates is linear. 

Let $\MC(M,\Sigma)$ be the mapping class group of $M$ fixing each zero of $\omega$. 
We choose a fundamental domain for the action of $\MC(M,\Sigma)$, and
think of the dynamics on the fundamental domain. Then, the $\SL_2(\reals)$ action becomes
\be\label{eq:SL2R-KZ}
x = \begin{pmatrix} \Re(\omega) \\ \Im(\omega) \end{pmatrix}
\mapsto hx = \begin{pmatrix} a & b \\ c & d \end{pmatrix} \begin{pmatrix} \Re(\omega) \\ \Im(\omega)
\end{pmatrix} A(h,x),
\ee
where $A(h,x) \in \Sp_{2g}(\zed) \ltimes \zed^{\noz-1}$
is the {\em Kontsevich-Zorich cocycle}. 

Thus, $A(h,x)$ is the change of basis one needs to perform to return the
point $hx$ to the fundamental domain. It can be interpreted as the
monodromy of the Gauss-Manin connection (restricted to the orbit of $\SL_2(\reals)$). 

\subsection{Affine measures and manifolds}\label{sec: affine inv mnfld}

For a subset $\mathcal E \subset \cH_1(\alpha)$ we write
\begin{displaymath}
\reals \mathcal E = \{ (M, t \omega) : (M,\omega) \in \mathcal E, t \in\reals \} \subset \cH(\alpha).
\end{displaymath}

An ergodic $\SL_2(\reals)$-invariant probability measure $\nu$ on
$\cH_1(\alpha)$ is called {\em affine} if the following hold:
\begin{itemize}
\item[(i)] The support $\cM$ of $\nu$ is an 
{\em immersed submanifold} of
  $\cH_1(\alpha)$, i.e., there exists a manifold $\cN$ and a proper continuous
  map $f: \cN \to \cH_1(\alpha)$ so that $\cM =
  f(\cN)$. The self-intersection set of $\cM$, i.e., the set of
  points of $\cM$ which do not have a 
  unique preimage under $f$, is a closed subset of $\cM$ of
  $\nu$-measure $0$.   Furthermore, each point in $\cN$ has a
  neighborhood $U$ such 
  that locally $\reals f(U)$ is given by a complex linear subspace defined
  over $\reals$ in the period coordinates.
\item[(ii)] Let $\bar\nu$ be the measure supported on $\reals
  \cM$ so that $d\bar\nu = 
    d\nu da$. Then each point in $\cN$ has a
      neighborhood $U$ such that the restriction of $\bar\nu$ to $\reals
      f(U)$ is an affine  linear measure in the period
    coordinates on $\reals f(U)$, i.e., it is (up to normalization) the
    restriction of the Lebesgue measure to the subspace $\reals f(U)$.
\end{itemize}

A suborbifold $\cM$ for which there exists a measure
$\nu$ such that the pair $(\cM, \nu)$ 
satisfies (i) and (ii) is said to be {\em affine invariant submanifold}.

We sometimes write $\nu_\cM$ to indicate the affine invariant measure $\nu$ on affine invariant submanifold $\cM$. 

Note that in particular, any affine invariant submanifold is a closed
subset of $\cH_1(\alpha)$ which is invariant under the action of $\SL_2(\reals)$, 
and which in period coordinates is an affine subspace.
We also consider the entire stratum {$\cH_1(\alpha)$} 
to be an (improper) affine invariant submanifold. 

\subsection{Typical affine stabilizer is trivial}
In this section, we prove the following statement:

\begin{propos}\label{prop:stabilizer-H-alpha}
Let $(\cM,\nu)\subset \cH_1(\alpha) \subset \Omega \tcM_{g,n}$ be an affine invariant 
submanifold. Assume that $\cM$ is {\em not} a Teichm\"{u}ller curve. 
Then for $\nu$-a.e.\ $x\in\cM$,
\[
\stab_{\SL_2(\reals)}(x) 
\]
is trivial. 
\end{propos} 

Recall that the set of self-intersections $\cM'$ of $\cM$ is a proper closed invariant submanifold of $\cM$, hence, $\dim\cM'<\dim \cM$, see~\cite{EMM-Orbit}; in particular, $\nu(\cM')=0$. Therefore, it suffices to prove the proposition for $\nu$-a.e.\ $x\in \cM\setminus\cM'$. 
Let $\widetilde \cM$ denote the lift of $\cM\setminus\cM'$ to $\Omega \mathcal T_g$. 

Fix $\phi \in \widehat \Mod(S_g)$ (that is, $\phi$ is not torsion). Define 
\[
P(\phi)= \Big\{ x \in \widetilde \cM :  
   \quad A \cdot x = \phi(x), \quad \text{for some $A \in \SL_2(\reals)$} \Big\}. 
\]
We will show, for every $\phi \in \widehat \Mod(S)$,  $P(\phi)$ is a $\nu$--measure zero 
subset of $\widetilde \cM$. Note that, by assumption, $\dim(\widetilde \cM) >3$. 

Consider $x \in P(\phi)$ and let $E_x$ be the $\GL^+(2, \R)$ orbit or $x$. 
Then $E_x$ can be considered as (an open subset of) the tangent space of the \Teich disk 
$\bbh_x$ associated to $x$ (the projection of $E_x$ to \Teich space). The restriction 
of \Teich metric to $\bbh_x$ equips $\bbh_x$ with the hyperbolic metric (up to 
a factor 2). We observe  that $\phi$ stabilizes $\bbh_x$ acting on $\bbh_x$ by
an isometry. In fact, we have either (see, for example, \cite[Lemma 5.6]{MT}) 
\begin{itemize} 
\item $\phi$ acts loxodromically on $\bbh_x$ and $\phi$ a pseudo-Anosov element.  
\item $\phi$ acts parabolically on $\bbh_x$ and $\phi$ is a multi-curve. 
\item $\phi$ acts elliptically on $\bbh_x$ and $\phi$ has finite order in $\Mod(S)$. 
\end{itemize} 
Note that the third case is excluded since we are assuming $\phi$ is not torsion. 
We argue each case separately.  

\subsection*{$\phi$ is pseudo-Anosov element} A pseudo-Anosov map $\phi$ 
stabilizes only one \Teich disk, the one where $\partial \bbh_x$ contains $F_+(\phi)$ 
and $F_-(\phi)$; the stable and the unstable foliation associated to $\phi$. 
Therefore, $P(\phi)= T_1 \bbh_x$, the unit tangent bundle over $\bbh_x$. 
Since $\cM$ is a not a \Teich curve, it has a dimension larger
than $3$. Hence $P(\phi)\cap \widetilde \cM$ is a $\nu$-measure zero subset of 
$\widetilde \cM$. 

\subsection*{$\phi$ is a multi-twist} Let $\phi$ be a multi-twist around $\gamma$, 
namely 
\[
\phi = \prod D_{\gamma_i}^{p_i}. 
\]
Let ${\R}P(\phi)$ be the subset of $\cH(\alpha)$ obtained from points in $P(\phi)$ 
after scaling. Then, for any $x \in {\R}P(\phi)$, a measured foliation that is 
topologically equivalent to $\gamma = \{ \gamma_1, \dots, \gamma_k)$ has to appear in 
the boundary of $\bbh_x$. That is, after a rotation, we can assume $x= (F_-, F_+)$ and 
$F_+ = \sum c_k \gamma_k$. 
Furthermore, $x$ has a cylinder decomposition where the modulus of these cylinders are 
rationally multiples of each other (\cite[Lemma 5.7]{MT}). That is, there are $r_i \in {\mathbb Q}$ such that 
\[
r_i \cdot \frac{i(F_-, \gamma_i)}{c_i} = r_j \cdot \frac{i(F_-, \gamma_j)}{c_j}, 
\]
for $1 \leq i, j \leq k$. We also have 
\[
\sum c_i \cdot i(F_-, \gamma_i)=\area(x). 
\]
That is, given $\gamma$, $r_i$, $F_-$ and $\area(x)$, we can calculate the values of 
$c_i$. Hence, $F_+$ and subsequently $x$ are uniquely determined by $\gamma$, $r_i$,
$F_-$ and $\area(x)$. There are countably many choices for the values $r_i$ and 
the multi-curve $\gamma$. We now show that the dimension of the space of possible 
measured foliations $F_-$ is half the dimension of ${\R}\widetilde \cM$
where ${\R}\widetilde \cM$ is the subset of $\cH(\alpha)$ obtained from point 
in $\widetilde \cM$ after scaling. 

For a filling bi-recurrent train-track $\tau$ (see \cite{PH} for definition and discussion)
any admissible weight on $\tau$ defines a measured foliation. We then say this 
measured foliation is carried by $\tau$. 
The complementary regions of a filling train tracks are $n$--gons or punctured $n$--gons. 
A foliation carried by $\tau$ has a singular point
associated to each complementary region of $\tau$. We say $\tau$ is of type
$\alpha=(\alpha_1,\ldots,\alpha_{\noz})$ if $\tau$ has $m$ complementary 
components that are punctured $\alpha_i$--gons, $i = 1 \dots m$. 
We denote the space of admissible weights in $\tau$ by $W(\tau)$.

\begin{lemma} 
For every $x \in \cH(\alpha)$ there are train tracks $\tau_+$ and $\tau_-$
of type $\alpha$ such that a neighborhood of $\cH(\alpha)$ around $x$ is 
homeomorphic to $U \times V$ where $U,V$ are open subsets of $W(\tau_+)$
and $W(\tau_-)$ respectively.  In fact, the real part of the period coordinates for 
$\cH(\alpha)$ give coordinates for $U$ and the imaginary part of the period 
coordinates, give coordinates for $V$.
\end{lemma}

\begin{proof}
Let $\Delta$ be a triangulation of $x$ by saddle connections 
(for example, $L^\infty$-Delanay triangulations see \cite[Section 3]{Ian}). 
Pick a subset $\mathcal B$ of the edges of $\Delta$ that give a basis for the 
homology of $x$ relative to the zeros $\Sigma$ of $x$. Then the complex numbers
$\{ \int_\omega x \} _{\omega \in \mathcal B}$ give local coordinates for $H(\alpha)$. 
For every edge $\omega$ of $\Delta$, we have 
\[
\I(\omega, F_-)= \Re \left(\int_\omega x\right).
\]
 In fact, $F_-$ can be constructed, triangle by triangle, from the set of real numbers 
$\{\I(\omega, F_-) \}_{\omega \in \Delta}$.  That is there is a train-track $\tau_-$
dual to the triangulation $\Delta$ (again, see \cite[Section 3]{Ian} for the 
construction of such train-tracks) such that 
$\{ \Re( \int_\omega x ) \} _{\omega \in \mathcal B}$ form an admissible weights
on $\tau_-$. At any point $y \in \cH(\alpha)$ near $x$, the triangulation $\Delta$ can still be 
represented by saddle connections and the set 
$\{ \Re( \int_\omega y ) \} _{\omega \in \mathcal B}$ form an admissible weights
on $\tau_-$ that is associated to the vertical foliation at $y$. That is, 
$\{ \Re( \int_\omega y ) \} _{\omega \in \mathcal B}$, thought of as admissible 
weights on $\tau_-$ give local cooridinates for the set of measured foliation that appear 
as a horizontal foliation of an element of $\cH_1(\alpha)$ near $x$. 
The same also holds for $\tau_+$ and the vertical foliations.  
\end{proof}

Since ${\R}\widetilde \cM$ is an affine sub-manifold 
of $\cH(\alpha)$, it is locally defined by 
a set of affine equations on period coordinates, see e.g. \S\ref{sec: affine inv mnfld} and~\cite{EM-Measure}. 
That is, there are subspaces $U' \subset U$ and $V' \subset V$, 
defined by the same set of affine equations, such that a neighborhood of $x$ in 
${\R}\widetilde \cM$ is naturally homeomorphic to $U' \times V'$. 
In particular, where $U'$ and $V'$ have half the dimension of ${\R}\widetilde \cM$. 

Let $W$ be the intersection of ${\R}P(\phi)$ with this neighborhood. Recall that, 
fixing the multi-curve $\gamma$, rational numbers $r_i$ and the area, 
every point in $W$ is determined, up to rotation, by a point in $U'$. 
Therefore, $W$ is a countable union of set of dimension $\dim(U') +2$. But 
\[
\dim(U') +2= \frac 12 \dim({\R}\widetilde \cM) +2 < \dim({\R}\widetilde \cM),
\] 
where the last inequality follows from the assumption that 
$\dim({\R}\widetilde \cM) >4$. 
That is, ${\R}P(\phi)\cap {\R}\widetilde \cM$ is a countable union of lower 
dimensional subset of ${\R}\widetilde \cM$ and therefore, has 
$\bar\nu$-measure zero, see \S\ref{sec: affine inv mnfld} for the definition of $\bar\nu$. 
Since, $\stab_{\SL_2(\reals)}(x)$ does not change
after scaling, we have, $P(\phi)\cap \widetilde \cM$ has $\nu$--measure zero in 
$\widetilde \cM$.

\subsection{Proof of Theorem~\ref{thm:main}}\label{sec: proof of main}
In this section we prove Theorem~\ref{thm:main}. The proof is based on the following proposition.

\begin{propos}\label{prop: BS conv stratum}
Let $\{E_k: k\in\bbn\}\subset \cH_1(\alpha) \subset \Omega \tcM_{g,n}$ be a sequence of closed $\SL_2(\bbr)$ 
orbits each equipped with the $\SL_2(\bbr)$-invariant probability measure $\mu_k$. Assume further that there exists 
an affine invariant submanifold $(\cM,\nu)\subset \cH_1(\alpha)$ so that 
\be\label{eq: muk to nu}
\mu_k\to \nu\quad\text{as $k\to\infty$}.
\ee
Let $V_k$ denote the Teichm\"uller curve associated to $E_k$
for all $k$. Then $\{V_k\}$ Benjamini-Schramm converges to $\bbh$.
\end{propos}

\begin{proof}
The proof if based on Proposition~\ref{prop:axiomatic-BS}. 
Let us write $E_k=\SL_2(\bbr).x_k$. 
We will show that $\SL_2(\bbr)/\stab_{\SL_2(\bbr)}(x_k)$ Benjamini-Schramm converges to $\SL_2(\bbr)$ 
from which the proposition follows. 

First note that $(\cM,\nu)$ is not a closed $\SL_2(\bbr)$ orbits, see~\cite[Thm.~2.3]{EMM-Orbit}. 
Hence, by Proposition~\ref{prop:stabilizer-H-alpha}, we have 
\be\label{eq: use prop stabilizer}
\text{$\stab_{\SL_2(\reals)}(x)=\{e\}$ for $\nu$-a.e.\ $x\in \mathcal M$.}
\ee

In the remaining pats of the argument, we write $H=\SL_2(\bbr)$ and use the notation in \S\ref{sec: BS conv general}. 
In particular, for all $R>0$, let
\[
B^H(e,R)=\{h\in H: \|h-I\|< R\text{ and }\|h^{-1}-I\|< R\}
\] 
where $\|\;\|$ denotes the maximum norm on ${\rm Mat}_2(\bbr)$ with respect to the standard basis.
Similarly, for $r>0$, let $B^H_{\rm Rie}(e,r)$ denote the ball of radius 
$r$ centered at the identity with respect the bi-${\rm SO}(2)$-invariant Riemannian metric on $H$ induced using the Killing form.

For every $x\in \cH_1(\alpha)$, let $r_x$ denote $1/2$ of the injectivity radius of 
$x$ in $\cH_1(\alpha)$ with respect to the Teichm\"uller metric. 
Then $x\mapsto r_x$ is continuous on $\cH_1(\alpha)$; moreover, $h\mapsto hx$ is injective on $B^H_{\rm Rie}(e,r_x)$. 

Let $R>0$ and for every $x\in\mathcal M$, put $B^H_R(x):= \bar B^H(e,R)\setminus B^H_{\rm Rie}(e,r_x)$; 
note that this a compact subset of $\SL_2(\bbr)$. Define $f_R:\mathcal M\to[0,\infty)$ by 
\[
f_R(x)=\min\Bigl\{\dist_{{\rm Teich}}(x,hx): h\in B^H_R(x)\Bigr\}.
\]
Note that $f_R$ is continuous. 
Indeed, let $y_m\to y$, and let $h_m\in B^H_R(y_m)$ be so that $f_R(y_m)=\dist_{{\rm Teich}}(y,h_my_m)$. 
Let $\{f_R(y_{m_i})\}$ be a converging subsequence of $\{f_R(y_m)\}$.  
Since $B^H_R(y_m)$ converges to $B^H_R(y)$ (in Hausdorff metric on compact sets), 
there is a subsequence $h_{m_{i_j}}\to h\in B^H_R(y)$ which implies: $f_R(y)\leq \lim_i f_R(y_{m_i})$. 
In consequence, $f_R(y)\leq\liminf f_R(y_m)$.    
To see the opposite direction, let $h\in B^H_R(y)$ be so that $f_R(y)=\dist_{{\rm Teich}}(y,hy)$. Let $h_m\in B^H_R(y_m)$
be so that $h_m\to h$, then $f_R(y_m)\leq \dist_{{\rm Teich}}(y,h_my_m)$ and for every $\vare>0$ we have 
$\dist_{{\rm Teich}}(y,h_my_m)\leq \dist_{{\rm Teich}}(y,hy)+\vare=f_R(y)+\vare$ so long as $m$ is large enough. 
Hence $\limsup f_R(y_m)\leq f_R(y)+\vare$. The continuity of $f_R$ follows.  

Moreover, in view of~\eqref{eq: use prop stabilizer}, we have $f_R(x)>0$ for $\nu$-a.e.\ $x\in\mathcal M$. Finally, since for every $x$, 
the map $h\mapsto hx$ is injective on $B^H_{\rm Rie}(e,r_x)$, we have $\stab_{\SL_2(\reals)}(x)\cap B^H_{\rm Rie}(e,r_x)=\{e\}$.
Thus if $f_R(x)>0$ for some $x\in\mathcal M$, then $\stab_{\SL_2(\reals)}(x)\cap B^H(e,R)=\{e\}$.

Altogether, we deduce that $f_R$ satisfies the conditions in Proposition~\ref{prop:axiomatic-BS}.
This and~\eqref{eq: muk to nu} imply that Proposition~\ref{prop:axiomatic-BS} applies and yields: 
\begin{quote}
$\SL_2(\bbr)/\stab_{\SL_2(\bbr)}(x_k)$ Benjamini-Schramm converges to $\SL_2(\bbr)$.
\end{quote}
The proof os complete.  
\end{proof}

\begin{proof}[Proof of Theorem~\ref{thm:main}]
Let $\{V_k: k\in\bbn\}\subset \tcM_{g,n}$ be a sequence of Teichm\"uller curves. 
We will show that for every subsequence $\{V_{k_i}\}$, there exists a further subsequence $\{V_{k_{i_j}}\}$ which 
Benjamini-Schramm converges to $\bbh$ the theorem follows from this.

Let $\{V_{k_i}\}$ be a subsequence of $\{V_k\}$. Passing to a further subsequence, which we continue to denote by $\{V_{k_i}\}$, 
we may assume that the corresponding $\SL_2(\bbr)$ orbits $\{E_{k_i}\}$ lie in $\cH_1(\alpha)\subset\Omega\tcM_{g,n}$ 
for a fixed $\alpha$. 

Now by~\cite[Thm.~2.3]{EMM-Orbit}, see also~\cite[Cor.~2.5]{EMM-Orbit}, 
there exists a subsequence $\{E_{k_{i_j}}\}$ of $\{E_{k_i}\}$, and an affine invariant manifold 
$(\mathcal M, \nu)$, so that $\mu_{k_{i_j}}\to \nu$ where $\mu_{k_{i_j}}$ 
denotes the $\SL_2(\bbr)$-invariant measure on $E_{k_{i_j}}$. 

By Proposition~\ref{prop: BS conv stratum}, we have $V_{k_{i_j}}$ Benjamini-Schramm converges to $\bbh$; as we wished to show.
\end{proof}

\bibliographystyle{amsplain}
\bibliography{papers}

\providecommand{\bysame}{\leavevmode\hbox to3em{\hrulefill}\thinspace}
\providecommand{\MR}{\relax\ifhmode\unskip\space\fi MR }
\providecommand{\MRhref}[2]{%
  \href{http://www.ams.org/mathscinet-getitem?mr=#1}{#2}
}
\providecommand{\href}[2]{#2}
\begin{thebibliography}{10}

\bibitem{7-Samurai}
Miklos Abert, Nicolas Bergeron, Ian Biringer, Tsachik Gelander, Nikolay
  Nikolov, Jean Raimbault, and Iddo Samet, \emph{On the growth of
  {$L^2$}-invariants for sequences of lattices in {L}ie groups}, Ann. of Math.
  (2) \textbf{185} (2017), no.~3, 711--790. \MR{3664810}

\bibitem{AG}
Artur Avila and S\'{e}bastien Gou\"{e}zel, \emph{Small eigenvalues of the
  {L}aplacian for algebraic measures in moduli space, and mixing properties of
  the {T}eichm\"{u}ller flow}, Ann. of Math. (2) \textbf{178} (2013), no.~2,
  385--442. \MR{3071503}

\bibitem{AGY}
Artur Avila, S\'{e}bastien Gou\"{e}zel, and Jean-Christophe Yoccoz,
  \emph{Exponential mixing for the {T}eichm\"{u}ller flow}, Publ. Math. Inst.
  Hautes \'{E}tudes Sci. (2006), no.~104, 143--211. \MR{2264836}

\bibitem{BFMS-Complex}
Uri Bader, David Fisher, Nicholas Miller, and Matthew Stover,
  \emph{Arithmeticity, superrigidity and totally geodesic submanifolds of
  complex hyperbolic manifolds}, 2020.

\bibitem{BFMS}
Uri Bader, David Fisher, Nick Miller, and Matthew Stover, \emph{Arithmeticity,
  superrigidity, and totally geodesic submanifolds}, 2019.

\bibitem{Corl}
Kevin Corlette, \emph{Archimedean superrigidity and hyperbolic geometry},
  Annals of Mathematics \textbf{135} (1992), no.~1, 165--182.

\bibitem{EM-Measure}
Alex Eskin and Maryam Mirzakhani, \emph{Invariant and stationary measures for
  the {${\rm SL}(2,\Bbb R)$} action on moduli space}, Publ. Math. Inst. Hautes
  \'{E}tudes Sci. \textbf{127} (2018), 95--324. \MR{3814652}

\bibitem{EMM-Orbit}
Alex Eskin, Maryam Mirzakhani, and Amir Mohammadi, \emph{Isolation,
  equidistribution, and orbit closures for the {${\rm SL}(2,\mathbb R)$} action
  on moduli space}, Ann. of Math. (2) \textbf{182} (2015), no.~2, 673--721.
  \MR{3418528}

\bibitem{Ian}
Ian Frankel, \emph{Cat(-1)-type properties for teichm\"uller space}, 2018.

\bibitem{Grom-Sch}
Mikhail Gromov and Richard Schoen, \emph{Harmonic maps into singular spaces
  andp-adic superrigidity for lattices in groups of rank one}, Publications
  Math{\'e}matiques de l'Institut des Hautes {\'E}tudes Scientifiques
  \textbf{76} (1992), no.~1, 165--246.

\bibitem{LW-Optimal}
Michael Lipnowski and Alex Wright, \emph{Towards optimal spectral gaps in large
  genus}, 2021.

\bibitem{Margulis-Arith-Inv}
G.~A. Margulis, \emph{Arithmeticity of the irreducible lattices in the
  semisimple groups of rank greater than {$1$}}, Invent. Math. \textbf{76}
  (1984), no.~1, 93--120. \MR{739627}

\bibitem{MM}
Gregory Margulis and Amir Mohammadi, \emph{Arithmeticity of hyperbolic
  3-manifolds containing infinitely many totally geodesic surfaces}, 2019.

\bibitem{MT}
Howard Masur and Serge Tabachnikov, \emph{Rational billiards and flat
  structures}, Handbook of dynamical systems, {V}ol. 1{A}, North-Holland,
  Amsterdam, 2002, pp.~1015--1089. \MR{1928530}

\bibitem{McMullen-SL2}
Curtis~T. McMullen, \emph{Dynamics of {${\rm SL}_2(\mathbb R)$} over moduli
  space in genus two}, Ann. of Math. (2) \textbf{165} (2007), no.~2, 397--456.
  \MR{2299738}

\bibitem{MMW-TeichCurves}
Curtis~T. McMullen, Ronen~E. Mukamel, and Alex Wright, \emph{Cubic curves and
  totally geodesic subvarieties of moduli space}, Ann. of Math. (2)
  \textbf{185} (2017), no.~3, 957--990. \MR{3664815}

\bibitem{Mozes-Shah}
Shahar Mozes and Nimish Shah, \emph{On the space of ergodic invariant measures
  of unipotent flows}, Ergodic Theory Dynam. Systems \textbf{15} (1995), no.~1,
  149--159. \MR{1314973}

\bibitem{PH}
R.~C. Penner and J.~L. Harer, \emph{Combinatorics of train tracks}, Annals of
  Mathematics Studies, vol. 125, Princeton University Press, Princeton, NJ,
  1992. \MR{1144770}

\end{thebibliography}

\providecommand{\bysame}{\leavevmode\hbox to3em{\hrulefill}\thinspace}
\providecommand{\MR}{\relax\ifhmode\unskip\space\fi MR }
\providecommand{\MRhref}[2]{%
  \href{http://www.ams.org/mathscinet-getitem?mr=#1}{#2}
}
\providecommand{\href}[2]{#2}

\end{document}